\documentclass[11pt]{amsart}
\usepackage{amssymb,latexsym,amsmath,cite}
\usepackage[linguistics]{forest}

\advance\textheight by \topskip   \numberwithin{equation}{section}

\newtheorem{theorem}{Theorem}[section]
\newtheorem{proposition}[theorem]{Proposition}
\newtheorem{corollary}[theorem]{Corollary}

\newtheorem{lemma}[theorem]{Lemma}

\newtheorem{example}[theorem]{Example}

\def\T{\mathcal{T}}
\def\P{\mathcal{P}_3}
\def\PP{\mathcal{P}}

\begin{document}

\title[]{An algorithmic approach based on generating trees for enumerating pattern-avoiding inversion sequences}
\author{Ilias Kotsireas}
\address{CARGO Lab, Wilfrid Laurier University, 75 University Avenue West, Waterloo, Ontario N2L 3C5, Canada}
\email{ikotsire@wlu.ca}
\author{Toufik Mansour}
\address{Department of Mathematics, University of Haifa, 3498838 Haifa, Israel}
\email{tmansour@univ.haifa.ac.il}
\author{G\"{o}khan Y\i ld\i r\i m}\thanks{G. Y\i ld\i r\i m was partially supported by Tubitak-Ardeb-120F352.}
\address{Department of Mathematics, Bilkent University, 06800 Ankara, Turkey}
\email{gokhan.yildirim@bilkent.edu.tr} 

\begin{abstract} We introduce an algorithmic approach based on generating tree method for enumerating the inversion sequences with various pattern-avoidance restrictions. For a given set of patterns, we propose an algorithm that outputs either an accurate description of the succession rules of the corresponding generating tree or an ansatz. By using this approach, we determine the generating trees for the pattern classes $I_n(000, 021), I_n(100, 021)$, $I_n(110, 021), I_n(102, 021)$, $I_n(100,012)$, $I_n(011,201)$, $I_n(011,210)$ and $I_n(120,210)$. Then we use the kernel method, obtain generating functions of each class, and find enumerating formulas. Lin and Yan studied the classification of the Wilf-equivalences for inversion sequences avoiding pairs of length-three patterns and showed that there are 48 Wilf classes among 78 pairs. In this paper, we solve six open cases for such pattern classes. Moreover, we extend the algorithm to restricted growth sequences and apply it to several classes. In particular, we present explicit formulas for the generating functions of the restricted growth sequences that avoid either \{12313,12323\}, $\{12313,12323,12333\}$, or $\{123\cdots\ell1\}$.
\end{abstract}
\maketitle
\section{Introduction}

An {\em inversion sequence} of length $n$ is an integer sequence $e=e_0e_1\cdots e_n$ such that $0\leq e_i\leq  i$ for each $0\leq i\leq n$. We denote by $I_n$ the set of inversion sequences of length $n$. There is a bijection between $I_n$ and $S_{n+1}$, the set of permutations of length $n+1$. Given any word $\tau$ of length $k$ over the alphabet $[k]:=\{0,1,\cdots,k-1\}$, we say that an inversion sequence $e\in I_n$ contains the pattern $\tau$ if there is a subsequence of length $k$ in $e$ that is order isomorphic to $\tau$; otherwise, we say that $e$ avoids the pattern $\tau$. For instance, $e=010213211 \in I_8$ avoids the pattern $201$ because there is no subsequence $e_je_ke_l$ of length three in $e$ with $j<k<l$ and $e_k<e_l<e_j$. On the other hand, $e=010213211$ contains the patterns $120$ and $0000$ because it has subsequence $---2-3-1-$  order isomorphic to 120, and subsequence $-1--1--11$ order isomorphic to $0000$. For a given pattern $\tau$, we use $I_n(\tau)$ to denote the set of all $\tau$-avoiding inversion sequences of length $n$. Similarly, for a given set of patterns $B$, we set $I_n(B)=\cap_{\tau \in B}I_n(\tau)$.  Pattern-avoiding permutation classes have been thoroughly studied by researchers for more than forty years; for some highlights of the results, see \cite{Kit} and references therein. A systematic study of pattern-avoidance for inversion sequences was initiated recently by Mansour and Shattuck \cite{ManS} for the patterns of length three with non-repeating letters and by Corteel et al. \cite{CMS} for repeating and non-repeating letters. Martinez and Savage \cite{MaSa} generalized and extended the notion of pattern-avoidance for the inversion sequences to triples of binary relations that lead to new conjectures and open problems. Many successfully studied research programs for permutations such as pattern-avoidance in terms of vincular patterns, pairs of patterns, and longer patterns have already been initiated to study for inversion sequences; for some recent results, see \cite{AuEl, BGRR, BBGR, CJL, Ch,HLi, ManS2, YanLin, LinFu, LinY} and references therein. In the context of inversion sequences, two sets of patterns $B_1$ and $B_2$ are considered Wilf equivalent if $|I_n(B_1)|=|I_n(B_2)|$ for all $n\geq 0$, that is, they have the same counting sequence. Note that there are thirteen patterns of length three up to order isomorphism; we denote them by $\P=\{000,001,010,100,011,101,110,021,012,102,120,201,210 \}$. Yan and  Lin \cite{YanLin} completed the classification of the Wilf-equivalences for inversion sequences avoiding pairs of length-three patterns. They showed that there are 48 Wilf classes among 78 pairs; for a complete list of the classes with open cases in terms of enumeration, see Table 1 and 2 in \cite{YanLin}. In this paper, we solve six open cases for such pattern classes: $I_n(000, 021)$, $I_n(102, 021)$, $I_n(100,012)$, $I_n(120,210)$, Wilf-equivalent $I_n(011,201)$ and $I_n(011,210)$, and Wilf-equivalent $I_n(100, 021)$ and $I_n(110, 021)$. Recently, Testart \cite{BT} also solved the following cases: $I_n(010,000),I_n(010,110),I_n(010,120)$, and the Wilf-equivalent pairs $I_n(010,201)$ and $I_n(010,210)$.

For simplicity of the notation, we leave curly brackets and use $I_n(\tau_1,\cdots,\tau_m)$ instead of $I_n(\{\tau_1,\cdots,\tau_m\})$ for a given set of patterns $B=\{\tau_1,\cdots,\tau_m\}$ throughout the paper. We shall use an algorithmic approach based on generating trees to enumerate pattern-restricted inversion sequences. {For some earlier results, in the context of pattern-restricted permutations, see \cite{V,Z} and references therein}. In this paper, we present applications of our algorithm only for the class $I_n(B)$ where either $B$ includes a single pattern or a pair of patterns of length three. However, the method applies to other inversion sequences with various pattern restrictions; for an application of the method to a pattern of length four, see \cite{Man23}. As we will see, the algorithm outputs either an accurate description of the succession rules of the generating tree for the given avoidance class or an ansatz based on which we can figure out the complete description of the generating tree. For most cases, we can use the kernel method \cite{Ker} to compute the generating functions and then obtain an exact enumerating formula for the corresponding pattern class or get a functional equation for the generating function. The latter case yields a procedure to calculate the coefficients of the generating function up to a given index.

We organize the paper as follows: In Section~\ref{GTA}, we present our algorithm and demonstrate how it works on some examples such as $B=\{000,001,012\}$ and $B=\{000,001\}$. In Section~\ref{caseB1}, we consider the open cases from single pattern of length three and obtain functional equations for the generating functions of $I_n(100)$, and Wilf-equivalent $I_n(201)$ and $I_n(210)$. In Section~\ref{caseB2}, we obtain the generating trees for the classes $I_n(000, 021), I_n(100, 021)$, $I_n(110, 021),$ $I_n(102, 021)$, $I_n(100,012)$, $I_n(011,201)$, $I_n(011,210)$ and $I_n(120,210)$ by using our algorithm. Then we use the kernel method, obtain the corresponding generating functions, and determine the counting sequences for them.
In the last section, we extend our algorithm to the restricted growth sequences; see the last section of the paper for definitions. We present explicit formulas for the generating functions for the number of restricted growth sequences of length $n$ that avoid either \{12313,12323\}, $\{12313,12323,12333\}$, or $\{123\cdots\ell1\}$.

\section{An algorithm based on generating trees}\label{GTA}
Any set $\mathcal{C}$ of discrete objects with a notion of a size such that for each $n$, there are finitely many objects of size $n$ is called a combinatorial class. A {\em generating tree} (see \cite{W}) for $\mathcal{C}$ is a rooted, labelled tree whose vertices are the objects of $\mathcal{C}$ with the following properties: (i) each object of $\mathcal{C}$ appears exactly once in the tree; (ii) objects of size $n$ appear at level $n$ in the tree (the root has level 0); (iii) the children of some object are obtained by a set of succession rules of the form that determines the number of children and their labels.

{Note that any pattern over the alphabet $[k]$ can be extended to an inversion sequence. Suppose a pattern $\tau=\tau_1\cdots\tau_m$ be given and let $\{0,1,\ldots,t\}$ denote the set of all letters appeared in $\tau$. We define $L_\tau$ to be the set of all inversion sequences $\theta^{(1)}\tau_1\theta^{(2)}\tau_2\cdots\theta^{(m)}\tau_m$ such that the length of the inversion sequence $\theta^{(1)}\tau_1\theta^{(2)}\tau_2\cdots\theta^{(j)}\tau_j$ is minimal for each $j=1,2,\ldots,m$. Note that some words $\theta^{(j)}$s might be empty. By the minimality condition on the lengths of $\theta^{(1)},\ldots,\theta^{(m)}$, we have that the length of any pattern in $L_\tau$ is at most $m+t$. For instance, if $\tau=021$, then $m=3$, $t=2$, and $L_\tau=\{0021,0121\}$; if $\tau=001$, then $m=3$, $t=1$, and $L_\tau=\{001\}$. Clearly, any inversion sequence $e$ avoids $B$ if and only if $e$ avoids $L=\cup_{\{\tau\in B\}}L_\tau$. For any set of patterns $B$, we identify $B$ with the set of patterns $L_B=\cup_{\{\tau\in B\}}L_\tau$. }

For a given set of patterns $B$, let $\mathcal{I}_B=\cup_{n=0}^{\infty} I_n(B)$. We will construct a pattern-avoidance tree $\T(B)$ for the class of pattern-avoiding inversion sequences $\mathcal{I}_B$. The tree $\T(B)$ is considered empty if no inversion sequence of arbitrary length avoids the set $B$. Otherwise, the root can always be taken as $0$, that is, $0\in \T(B)$. Starting with this root which stays at level $0$, the remainder of the  tree $\T(B)$ can then be constructed in a recursive manner such that the $n^{th}$ level of the tree consists of exactly the elements of $I_n(B)$ arranged in such a way that the parent of an inversion sequence $e_0e_1\cdots e_n \in I_n(B)$ is the unique inversion sequence $e_0e_1\cdots e_{n-1}\in I_{n-1}(B)$. The children of $e_0e_1\cdots e_{n-1}\in I_{n-1}(B)$ are obtained from the set $\{e_0e_1\cdots e_{n-1}e_n\mid e_n=0,1,\ldots,n\}$ by applying the pattern-avoiding restrictions of the patterns in $B$. We arrange the nodes from the left to the right so that if $e=e_0e_1\cdots e_{n-1}i$ and $e'=e_0e_1\cdots e_{n-1}j$ are children of the same parent $e_1\cdots e_{n-1}$, then $e$ appears on the left of $e'$ if $i<j$. See Figure \ref{figT1} for the first few levels of $\T(\{012\})$. Note that the size of $I_n(B)$ equals the number of nodes in the $n$-th level of $\T(B)$.
	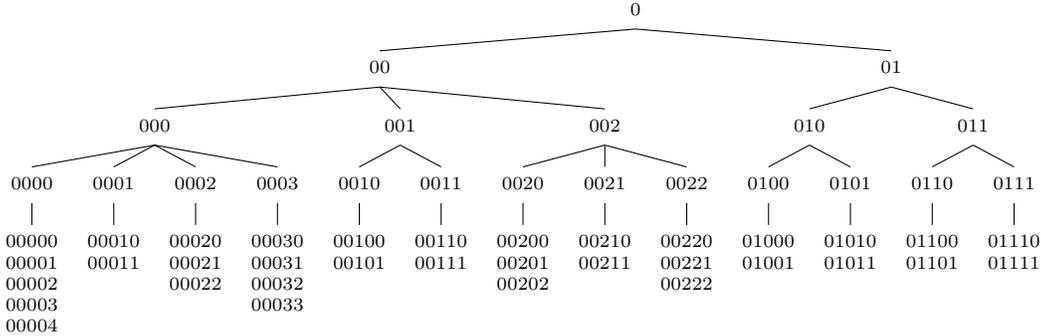
\begin{figure}[htp]
{\tiny
\begin{forest}
for tree={fit=band,}
[0[00,[000,[0000,[00000\\00001\\00002\\00003\\00004]] [0001,[00010\\00011]] [0002,[00020\\00021\\00022]] [0003,[00030\\00031\\00032\\00033]]] [001,[0010,[00100\\00101]] [0011,[00110\\00111]]] [002,[0020,[00200\\00201\\00202]], [0021,[00210\\00211]],[0022,[00220\\00221\\00222]]]] [01,[010,[0100,[01000\\01001]] [0101,[01010\\01011]]] [011,[0110,[01100\\01101]] [0111,[01110\\01111]]]]]
 \end{forest}}
\caption{First four levels of $\T(\{012\})$}\label{figT1}
\end{figure}

For a given set of patterns $B$, it plays an essential role to understand the nature of the tree $\T(B)$ to enumerate the class $\mathcal{I}_B=\cup_{n=0}^{\infty} I_n(B)$. Let $\T(B;e)$ denote the subtree consisting of the inversion sequence $e$ as the root and its descendants in $\T(B)$. In our arguments, it will be important to determine if the subtrees starting from two distinct nodes $e, e' \in \T(B)$ are isomorphic or not, that is, $\T(B;e)\cong\T(B;e')$ in the sense of plane tree isomorphism. Lemma~\ref{lem1} provides an easy to check criteria for this task, for similar results, see \cite{BM}.
\begin{lemma}\label{lem1} Let $t$ be the length of the longest pattern in $B$.
We have that $\T(B;e)\cong\T(B;e')$ for two inversion sequence $e, e' \in \T(B)$ if and only if $\T^{2t}(B;e)\cong\T^{2t}(B;e')$ where $\T^m(B;e)$ denotes the finite tree corresponding to the first $m-1$ level of $\T(B; e)$.
\end{lemma}
\begin{proof}
{Since avoiding $B$ is equivalent to avoiding $L_B=\cup_{\{\tau\in B\}}L_\tau$ in the set of inversion sequences, we assume that any pattern in $B$ is an inversion sequence.

Let $e, e' \in \T(B)$. Clearly, $\T(B;e)\cong\T(B;e')$ implies $\T^{2t}(B;e)\cong\T^{2t}(B;e')$.

Now, let us assume that $\T(B;e)\not\cong\T(B;e')$ as plane trees. We read the nodes of $\T(B;e)$ (resp. $\T(B,e')$) from top to bottom and from left to right and denote them as $e_j$ (resp. $e'_j$) with $e_0=e$ (resp. $e'_0=e'$). Since $\T(B;e)\not\cong\T(B;e')$, there exists $s\geq0$ minimal such that (1) the number children of $e_j$ equals the number of children of $e'_j$, for $j=1,2,\ldots,s-1$, and (2) the number of children of $e_s$ does not equal the number of children of $e'_s$. By construction of $\T(B)$, for all $j=1,2,\ldots,s-1$, there exist letters $p_{ij},q_{ij}$ such that the inversion sequences $ef_j:=ep_{1j}p_{2j}\cdots p_{i_jj}$ and $e'f'_j:=e'q_{1j}q_{2j}\cdots q_{i_jj}$ avoid $B$,  while the inversion sequence $ef_s:=ep_1p_2\cdots p_{i_s}$ contains $\tau\in B$; the inversion sequence $e'f_s'=e'q_1q_2\cdots q_{i_s}$ avoids $B$ and there exists a bijection $\alpha$ such that $q_j=\alpha(p_j)$, for all $j=1,2,\ldots,i_s$.

Any occurrence of $\tau$ in $ef_s$ can use at most $t-1$ letters of $f_s$.
Thus, there is a subsequence $g=p_{k_1}\cdots p_{k_m}$ of $f_s$ of minimal length $m$ such that the word $eg$ contains $\tau$ and $m\leq t-1$ but the word $e'g'$ avoids $B$, where $g'=q_{k_1}\cdots q_{k_{t-1}}$ is a subsequence of $f'_s$.

Since $ef_s$ is an inversion sequence, then there exists inversion sequence $e\tilde{g}\in L_{eg}$ such that $e\tilde{g}$ is a subsequence of $ef_s$. Since each letter $p_j$ is mapped to the letter $q_j$ by $\alpha$, we see that the sequence $\tilde{g}$ is mapped to $\tilde{g}'$. Since $\tilde{g}'$ is subsequence of $f'_s$, we see that $e'\tilde{g}'$ is an inversion sequence in $L_{e'g'}$.

Thus, the inversion sequence $e\tilde{g}$ contains $\tau$ such that the length of $\tilde{g}$ is at most $t+t-1$ and the inversion sequence $e'\tilde{g}'$ avoids $\tau$ such that the length of $\tilde{g}'$ is at most $t+t-1$. Hence, $\T^{2t}(B;e)\not\cong\T^{2t}(B;e')$.}
\end{proof}

We define an equivalence relation on the set of nodes of $\T(B)$ as follows. Let $v=v_0v_1\cdots v_a$ and $w=w_0w_1\cdots w_b$ two nodes in $\T(B)$. We say that $v$ is equivalent to $w$, denoted by $v\sim w$, if and only if $\T(B;v)\cong\T(B;w)$. Note that Lemma \ref{lem1} performs a finite procedure for checking $v\sim w$. Define $V[B]$ to be the set of all equivalence classes in the quotient set $\T(B)/\sim$. We will represent each equivalence class $[v]$ by the label of the unique node $v$ which appears on the tree $\T(B)$ as the left-most node at the lowest level among all other nodes in the same equivalence class. Let $\T[B]$ be the same tree $\T(B)$ where we replace each node $v$ by its equivalence class label, see Figure~\ref{figT3}. That is, $w$ is relabelled by $v$ such that
\begin{itemize}
\item $v\sim w$, and
\item either $a<b$ or $a=b$ such that, in the list of the nodes at level $a$ in the tree $\T(B)$ from left to right, the node $v$ does appear before the node $w$.
\end{itemize}

Next, we define an algorithm for finding $\T[B]$ for a given set of patterns $B$ with $0\not\in B$. {As we run the algorithm, we use the set $Q_D$ and $R$ to keep track of the equivalence classes that are obtained at step $D$ and the deduced succession rules for the generating tree that are obtained up to step $D$, respectively}. The details are as follows:
\begin{itemize}
\item[(1)] We {\bf initialize} the tree $\T[B]$ by the root $0$, and define $Q_0=\{0\}$ and $R=\emptyset$.
\item[(2)] Let $D$ be any positive integer.

\item[(3)] For all $i=1,2,\ldots,D$,
  \begin{itemize}
  \item[(3.1)] for any $w\in Q_{i-1}$, we denote the set of all children of $w$ in $\T(B)$ by $N_w$. We {\bf denote} the set of all children of all new equivalence classes at $i^{th}$ step by $M_i=\cup_{w\in Q_{i-1}}N_w$. If $M_i=\emptyset$, then we stop the loop and go to (4).

  \item[(3.2)] we {\bf initialize} the set $Q_i$ (set of new equivalence classes at $i^{th}$ step) to be empty set. For each child $w$ in $M_i$,
    \begin{itemize}
    \item[(3.2.1)] we find $v\in \cup_{j=0}^{i-1}Q_j$, if possible, such that $w\sim v$, where we use Lemma \ref{lem1} to check that $w\sim v$ holds or not;
    \item[(3.2.2)] otherwise, we add the equivalence class $w$ to $Q_i$.
    \end{itemize}
  \item[(3.3)] based on (3.2), we {\bf add} the rule $w\rightsquigarrow v_1v_2\cdots v_s$ to the set $R$, where {$v_j$ is the label of the $j^{th}$ child of $w$}, from left to right, in $\T[B]$.
  \end{itemize}
\item[(4)] If we {\bf stop} at $(3.1)$, then we have the finite set of labels $\cup_{j=0}^{i-1}Q_j$ and finite set of succession rules $R$ that specifies the tree $\T[B]$ with the root $0$. In this context, $B$ is called {\em regular}.
\item[(5)] Otherwise, we have set of succession rules $R$ that specifies the tree $\T[B]$ with its root $0$ up to level $k(D)$ where $k(D)$ is an integer depending on $D$. We could {\bf guess}, if possible, all the set of succession rules of $\T[B]$ based on $R$, then use Lemma \ref{lem1} to prove this claim. In case we fail to guess the whole set of the succession rules, then either we increase $D$ or we say that our procedure does not lead us to determine the all succession rules of $\T[B]$.
\end{itemize}
We will use the following fact throughout the paper: for any pattern collection $B$, $\T(B)\cong\T[B]$ (as plane trees) and the number of nodes at the $n^{th}$ level of the generating tree is equal to the number of inversion sequence of length $n$ avoiding the patterns in $B$.
\begin{example}
Let $B=\{000,001,012\}$, we apply our procedure with $D=5$ as follows:
\begin{center}
\begin{tabular}{l||l|l|l|l}
$i$&$M_i$&Comments&$Q_i$&$R$\\\hline\hline
$0$& & &$\{0\}$&$\emptyset$\\\hline
$1$&$\{00,01\}$&$00\not\sim0$, $01\not\sim0$, $01\not\sim00$  &$\{00,01\}$     &$\{0\rightsquigarrow00,01\}$\\
$2$&$\{010,011\}$&$010\sim00$, $011\not\sim v\in Q_0\cup Q_1$&$\{011\}$&
$\{0\rightsquigarrow00,01$, $01\rightsquigarrow00,011\}$\\\hline
$3$&$\{0110\}$&$0110\sim00$&$\emptyset$&$\{0\rightsquigarrow00,01$\\ &&&&$01\rightsquigarrow00,011$, $011\rightsquigarrow00\}$.
\end{tabular}
\end{center}
Hence, the generating tree $\T[B]$ given by the algorithm has the following succession rules:
\begin{align*}
&\mbox{Root: }0,\,\mbox{Rules: }
0\rightsquigarrow00,01,\quad
01\rightsquigarrow00,011,\quad
011\rightsquigarrow00.
\end{align*}
We want to find the generating function $R(x)=\sum_{n\geq0}|I_n(B)|x^{n+1}$. We use $A_w(x)$ to denote the generating function for the number of nodes in the subtree $\T(B;w)$. Hence, by the generating tree $\T[B]$, we have $R(x)=x+xA_{00}(x)+xA_{01}(x)$, $A_{00}(x)=x$, $A_{01}(x)=x+xA_{00}(x)+xA_{011}(x)$, and $A_{011}(x)=x+xA_{00}(x)$. By solving for $R(x)$, we obtain that $R(x)=x^4+2x^3+2x^2+x$.
\end{example}

\begin{example}
Let $B=\{000,001\}$. By applying our procedure with $D=5$, we guess that the tree $\T[B]$ is given by
\begin{align*}
&\mbox{Root: }a_0,\,\mbox{Rules: }
a_0\rightsquigarrow b_0a_1,\quad
a_m\rightsquigarrow b_0b_1b_2\ldots b_ma_{m+1}\quad
b_m\rightsquigarrow b_0b_1b_2\ldots b_{m-1},
\end{align*}
where $a_0=0$, $a_m=012\cdots m$, $b_0=00$ and $b_m=012\cdots(m-1)mm$ for $m\geq 1$.
We will make use of Lemma \ref{lem1} to verify the succession rules of the generating tree. Since other cases are very similar, we only show that the succession rule $a_m\rightsquigarrow b_0b_1b_2\ldots b_ma_{m+1}$ holds. Let $v=012\cdots m$, then the children of $v$ in $\T(B)$ are
$012\cdots mj$ with $j=0,1,\ldots,m+1$. By using Lemma \ref{lem1}, we see that $012\cdots m0\sim00$, $012\cdots mj\sim012\cdots(j-1)jj$ with $j=1,2,\ldots,m$, and for $j=m+1$ we have a new equivalence class $012\cdots m(m+1)$. This verifies the succession rule $a_m\rightsquigarrow b_0b_1b_2\ldots b_ma_{m+1}$.

We aim to compute the generating function $R(x)=\sum_{n\geq0}|I_n(B)|x^{n+1}$. We use $A_w(x)$ to denote the generating function for the number of nodes in the subtree $\T(B;w)$. Let us define $B_m(x)=A_{012\cdots m}(x)$ and $C_m(x)=A_{012\cdots(m-1)mm}(x)$, for $m\geq1$. Then, by generating tree $\T[B]$, we obtain that $R(x)=x+xA_{00}(x)+xA_{01}(x)$, $A_{00}(x)=x$, and
\begin{align*}
B_m(x)&=x+x^2+x(C_1(x)+\cdots+C_m(x))+xB_{m+1}(x),\\
C_m(x)&=x+x^2+x(C_1(x)+\cdots+C_{m-1}(x)),
\end{align*}
We define $G(x,u)=\sum_{m\geq1}G_m(x)u^{m-1}$, where $G\in\{B,C\}$. Hence, by multiplying the recurrence relations by $u^{m-1}$ and summing over $m\geq1$, we have
\begin{align}
B(x,u)&=\frac{x(1+x)}{1-u}+\frac{x}{1-u}C(x,u)+\frac{x}{u}(B(x,u)-B(x,0)),\label{eqex21}\\
C(x,u)&=\frac{x(1+x)}{1-u}+\frac{x}{1-u}C(x,u)-xC(x,u).\label{eqex22}
\end{align}
By solving \eqref{eqex22} for $C(x,u)$, we have
$$C(x,u)=\frac{x(1+x)}{(1+x)(1-u)-x}.$$
The equations of type \eqref{eqex21} can be solved systematically using the kernel method \cite{Ker}. In this case, if we assume that $u=x$, then
\begin{align}
B(x,0)=\frac{x(1+x)}{1-x}+\frac{x}{1-x}C(x,x)
=\frac{x(1+x)}{1-x-x^2}.
\end{align}
Hence, by comparing coefficients of $x^{n+1}$, we obtain that $|I_n(000,001)|=Fib_{n+2}$, where $Fib_n$ is the $n^{th}$ Fibonacci number, that is, $Fib_n=Fib_{n-1}+Fib_{n-2}$ with $Fib_0=0$ and $Fib_1=1$.
\end{example}

\section{Set of patterns $B \subset \P$ with $|B|=1$}\label{caseB1}
As we discussed in the introduction, the first systematic study of pattern-avoiding inversion sequences was carried out for the case of a single pattern of length three in \cite{CMS} and \cite{ManS}. The results of these papers demonstrated that there are some remarkable connections with pattern-restricted inversion sequences and other well-studied combinatorial structures. Some of the highlights of their results can be summarized as follows: the odd-indexed Fibonacci numbers count $I_n(012)$, the large Schröder numbers count $I_n(021)$, the Euler up/down numbers count $I_n(000)$, the Bell numbers count $I_n(011)$, and powers of two count $I_n(001)$; for details see the above references. There are still no enumerating formulas for the avoidance sets $I_n(100)$ and $I_n(120)$, and Wilf-equivalent $I_n(201)$ and $I_n(210)$. For the enumeration of the pattern $010$, see the recent preprint \cite{BT}. In this section, we use this paper's algorithmic approach and derive functional equations for the generating functions of the classes $I_n(100)$ and $I_n(201)$. For similar results in the context of pattern-restricted permutations, see \cite{NZ,YZ}.

\subsection{Class $I_n(100)$}\label{case100}
From now, we denote the constant word
$kk\cdots k$ of length $d$ by $k^d$, for any letter $k$ and positive integer $d$. Our algorithm allows us to guess the generating tree $\T[\{100\}]$.
\begin{theorem}\label{th100tt}
The generating tree $\T[\{100\}]$ is given by
$$\mbox{Root: }a_1,\quad \mbox{Rules: }
a_m\rightsquigarrow a_{m+1}b_{m,1}\cdots b_{m,m},\,\,
b_{m,j}\rightsquigarrow (b_{m,j-1})^j b_{m+1,j}\cdots b_{m+1,m+1},$$
where $a_m=0^m$ and $b_{m,j}=0^mj$, for all $1\leq j\leq m$.
\end{theorem}
\begin{proof}
We proceed by using our algorithm. We label the inversion sequences $0\in I_0$ by $0$. Clearly, the children of $0^m$ are $0^{m+1},0^m1,\ldots,0^mm$. Thus it remains to show that the children of $0^mj$ are $(0^m(j-1))^j(0^{m+1}j)\cdots(0^{m+1}(m+1))$. Let $v_i=0^mji$, we have that
\begin{itemize}
\item  if $i=0,1,\ldots,j-1$, then $v_i\sim0^m(j-1)$ in $\T(\{100\})$ {(by removing the letter $i$ and subtracting $1$ from each letter bigger than $i$)}.
\item if $i=j,j+1,\ldots,m$, then $v_i=0^{m}ji\sim0^{m+1}i$ in $\T(\{100\})$ {(by replacing the letter $j$ by $0$)},
\end{itemize}
which completes the proof.
\end{proof}
To study the generating function $R(x)=\sum_{n\geq0}|I_n(100)|x^{n+1}$, we define $A_m(x)$ and $B_{m,j}(x)$ to be the generating functions for the number of nodes in the subtrees $\T(B;0^m)$ and $\T(B;0^mj)$, respectively. Let $B_m(x)=\sum_{j=1}^mB_{m,j}(x)$ and $B_{m,0}(x)=A_{m+1}(x)$. Thus, from the generating tree's succession rules, we get
\begin{align*}
A_m(x)&=x+xA_{m+1}(x)+xB_m(x),\quad m\geq1,\\
B_{m,j}(x)&=x+jxB_{m,j-1}(x)+xB_{m+1,j}(x)+\cdots+xB_{m+1,m+1}(x),\quad j=1,2,\ldots,m.
\end{align*}

Then, we define the following bivariate generating functions: $A(x,v)=\sum_{m\geq1}A_m(x)v^{m-1}$, $B_m(x,u)=\sum_{j=1}^mB_{m,j}(x)u^{m-j}$, and $B(x,v,u)=\sum_{m\geq1}B_m(x,u)v^{m-1}$. Note that the system of recurrences can be written as follows:
\begin{align*}
A(x,v)&=\frac{x}{1-v}+\frac{x}{v}(A(x,v)-A(x,0))+xB(x,v,1),\\
B(x,v,u)&=\frac{x}{(1-v)(1-vu)}
-xu\frac{\partial}{\partial u}\frac{B(x,v,u)-B(x,v,0)}{u}
-xu\frac{\partial}{\partial u}\frac{A(x,uv)-A(x,0)}{uv} \\ &+x\frac{\partial}{\partial v}(\frac{v}{u}(B(x,v,u)-B(x,v,0))+\frac{A(x,uv)-A(x,0)}{u})\\
&+\frac{x}{uv(1-u)}(B(x,v,u)-uB(x,uv,1)-(1-u)B(x,0,0))\\
&-\frac{x}{uv}(B(x,v,0)-B(x,0,0)).
\end{align*}
By taking $v=x$ into first equation, we get the following result.
\begin{theorem}\label{thm100}
The generating function $\sum_{n\geq0}|I_n(100)|x^{n+1}$ is equal to $A(x,0)$ that satisfies the following functional equation: $$A(x,0)=\frac{x}{1-x}+xB(x,x,1).$$
\end{theorem}
We can not solve this functional equation to obtain an explicit expression for the generating function. But the functional equation can used to obtain the first $n$ terms of the generating function $A(x,0)$ for any positive integer $n$. The first 24 terms are
1, 2, 6, 23, 106, 565, 3399, 22678, 165646, 1311334, 11161529, 101478038, 980157177, 10011461983, 107712637346, 1216525155129, 14380174353934, 177440071258827, 2280166654498540, 30450785320307436, 421820687108853017, 6050801956624661417, 89738550379292147192, 1374073440225390131037, 21694040050913295537753.

\subsection{Class $I_n(201)$ or $I_n(210)$}\label{case201}
Based on the algorithm's ansatz, we get the same succession rules for the generating trees of $I_n(201)$ and $I_n(210)$. The generating tree is given as follows {(Since the similarity to proof of Theorem \ref{th100tt}, from now we state the generating trees without proofs)}:
\begin{align*}
\mbox{Root: }a_1,\quad\mbox{Rules: }&a_m\rightsquigarrow a_{m+1}a_{m+1}b_{m,2}b_{m,3}\cdots b_{m,m},\\
&b_{m,j}\rightsquigarrow a_{m+3-j}b_{m+3-j,2}\cdots b_{m+1,j}b_{m+1,j}b_{m+1,j+1}\cdots b_{m+1,m+1},
\end{align*}
where $a_m=0^m$ and $b_{m,j}=0^mj$, for all $m\geq1$ and  $2\leq j\leq m$.

This result implies the following corollary:

\begin{corollary} $|I_n(201)|=|I_n(210)|$ for all $n\geq 1$.
\end{corollary}
For a bijection between these two classes, see \cite{ManS}.

To study the generating function $R(x)=\sum_{n\geq0}|I_n(201)|x^{n+1}$, we define $A_m(x)$ and $B_{m,j}(x)$ to be the generating functions for the number of nodes in the subtrees $\T(B;0^m)$ and $\T(B;0^mj)$, respectively. Let $B_m(x)=\sum_{j=2}^mB_{m,j}(x)$. Thus,
\begin{align}
A_m(x)&=x+2xA_{m+1}(x)+xB_m(x),\, m\geq1,\label{eq201a1}\\
B_{m,j}(x)&=x+xA_{m+3-j}(x)+x\sum_{i=2}^jB_{m+1-j+i,i}(x)+x\sum_{i=j}^{m+1}B_{m+1,i}(x),\, 1\leq j\leq m.\label{eq201a2}
\end{align}
Clearly, $A_1(x)=x+2xA_2(x)$.

We define the generating functions: $A(x,v)=\sum_{m\geq1}A_m(x)v^{m-1}$, $B_m(x,u)=\sum_{j=2}^mB_{m,j}(x)u^{m-j}$, and $B(x,v,u)=\sum_{m\geq2}B_m(x,u)v^{m-2}$. Then by multiplying \eqref{eq201a1} by $v^{m-1}$ and summing over $m\geq1$, we obtain
\begin{align}
A(x,v)&=\frac{x}{1-v}+\frac{2x}{v}(A(x,v)-A(x,0))+xvB(x,v,1).\label{eq201a3}
\end{align}
Note that $A_2(x)=(A(x,0)-x)/(2x)$.

By multiplying \eqref{eq201a2} by $u^{m-j}v^{m-1}$ and summing over $2\leq j\leq m$, we obtain
\begin{align*}
B(x,v,u)&=\frac{x}{(1-v)(1-vu)}
+\frac{x}{u^2v^2(1-v)}(A(x,uv)-\frac{uv}{2x}(A(x,0)-x)-A(x,0))\\
&+\frac{x}{uv(1-v)}(B(x,v,u)-B(x,v,0))\\
&+\frac{x}{uv(1-u)}(B(x,v,u)-uB(x,uv,1))-\frac{x}{uv}B(x,v,0).
\end{align*}
Hence, by setting $v=2x$ into \eqref{eq201a3}, we obtain the following functional equation for the generating function. For a similar functional equation, see \cite{ManS}.
\begin{theorem}\label{thm201}
The generating function $\sum_{n\geq0}|I_n(201)|x^{n+1}$ is equal to $A(x,0)$ that satisfies the following functional equation:
$$A(x,0)=\frac{x}{1-2x}+2x^2B(x,2x,1).$$
\end{theorem}
By using the above theorem, we can obtain the first $n$ terms of the generating function $A(x,0)$ for any positive integer $n$. The first 24 terms are
1, 2, 6, 24, 118, 674, 4306, 29990, 223668, 1763468, 14558588, 124938648, 1108243002, 10115202962, 94652608690, 905339525594, 8829466579404, 87618933380020, 883153699606024, 9028070631668540, 93478132393544988, 979246950529815364, 10368459385853924212, 110866577818487410864.

%1196244348486123545848.

\section{Set of patterns $B \subset \P$ with $|B|=2$}\label{caseB2}
Inversion sequences avoiding pairs of patterns of length three was first systematically studied by Yan and Lin \cite{YanLin}. They determined that there are 48 Wilf classes among 78 pairs and provided enumerating formulas for some of the classes; for a complete list see Table 1 and 2 in \cite{YanLin}. In this section, we first obtain the generating trees corresponding to the classes $I_n(000, 021), I_n(100, 021)$, $I_n(110, 021),$ $I_n(102, 021)$, $I_n(100,012)$, $I_n(011,201)$, $I_n(011,210)$ and $I_n(120,210)$ by using our algorithm. It will follow from the generating trees that classes $I_n(011,201)$ and $I_n(011,210)$ are Wilf-equivalent, and $I_n(100, 021)$ and $I_n(110, 021)$ are Wilf-equivalent. Then we use the kernel method and determine the counting sequences for them, see Table \ref{tb00}. For some additional new results in this direction, see \cite{BT}.
\begin{center}
\begin{table}[htp]
\begin{tabular}{c|c|c}
\hline
\hline
$ B$ & $a_n=|I_n(B)|$ & reference \\
\hline
\hline
&&\\[-6pt]
(000,021) & $\frac{1}{2}(3a_{n-1}+a_n-3a_{n+1}+a_{n+2})$ &  Theorem~\ref{thAA2}\\
&\footnotesize{$a_n=\sum_{k=0}^n(-1)^{n-k}\binom{n}{k}\binom{2k}{k}$}&\\[2pt]
\hline
&&\\[-6pt]
(100,021)$\sim$(110,021) & $\frac{n^2+n+6}{2(n+3)(n+2)}\binom{2n+2}{n+1}$ &  Theorem~\ref{thCC3}\\[2pt]
%&&\\
\hline
&&\\[-6pt]
(102,021) & $\sum_{k=0}^n\frac{1}{k+1}\binom{2k}{k}-1-\frac{1}{6}n^3-\frac{11}{6}n+2^n$  & Theorem~\ref{thDD1}\\[2pt]
%&&\\
\hline
&&\\[-6pt]
 (100,012)& $\frac{(n+7)Fib_n+15Fin_{n+1}+nFib_{n+2}}{5}-1-\binom{n+2}{2}$ & Theorem~\ref{thBB2} \\[2pt]
%&&\\
\hline
(011,201)$\sim $(011,210) & functional equation & Theorem~\ref{thm011201} \\
&for the generating function&\\
\hline
(120,210) &  functional equation &  Theorem~\ref{thm120210}\\
&for the generating function&\\
\hline
\end{tabular}
\caption{Summary of the results}\label{tb00}
\end{table}
\end{center}

\subsection{Class $I_n(000,021)$}\label{sec000-021}
Let $B=\{000,021\}$. When we apply our algorithm to the pattern class $I_n(B)$, we obtain a generating tree that leads to an enumerating formula for this case.

We define $r_0=0$, $a_m=0011\cdots mm, b_m=0011\cdots(m-1)(m-1)m$ with $m\geq0$, and $c_m=01122\cdots mm, d_m=01122\cdots(m-1)(m-1)m$ with $m\geq1$. The generating tree $\mathcal{T}[B]$ is given by
\begin{align*}
\mbox{Root: }r_0,\,\,\mbox{Rules: }& r_0\rightsquigarrow a_0d_1,\,\, a_m\rightsquigarrow b_{m+1}b_m\cdots b_0,\,\,b_m\rightsquigarrow a_m b_m b_{m-1}\cdots b_0,\quad m\geq0,\\
&c_m\rightsquigarrow a_m d_{m+1}d_m\cdots d_1,\,\,d_m\rightsquigarrow b_m c_m d_m d_{m-1}\cdots d_1,\quad m\geq1.
\end{align*}
This result follows from the following observations. We label the inversion sequences $0\in I_0$ and $00,01\in I_1$ by $r_0$ and $a_0,d_1$, respectively. Thus, $r_0\rightsquigarrow a_0d_1$. It remains to show that the generating tree's succession rules hold. Since the other cases are very similar, we will verify only the rule $a_m\rightsquigarrow b_{m+1}b_m\cdots b_0$ for all $m\geq0$. Let $e=e_0e_1\cdots e_n$ be any inversion sequence labelled by $a_m$. So, by definitions, we have that $\T(B;e)\cong\T(B;a_m)$. On other hand, the inversion sequence that labelled by $a_m j=0011\cdots mmj$ where $j=m+1,m+2,\ldots,2m+2$ (otherwise, $a_mj$ does not avoid $B$). Moreover, (i) $a_m(m+1)=0011\cdots mm(m+1)=b_{m+1}$; (ii) $a_m(m+j)=0011\cdots mm(m+j)$; the subtree $\T(B;a_m(m+j))$ is isomorphic to the subtree $\T(B;b_{m+2-j})$ by removing the letters $m+2-j,m+3-j,\ldots,m$ and decreasing each letter greater than $m$ by $2j-1$.
Thus, Lemma \ref{lem1} gives the children of the node with label $a_m$ are exactly the nodes labelled by $b_{m+1},b_m,\ldots,b_0$, that is, $a_m\rightsquigarrow b_{m+1}b_m\cdots b_0$ with $m\geq0$.

In order to find an explicit formula for the generating function for the number of inversion sequences in $I_n(B)$, we define
$R(x)$ (respectively, $A_m(x)$, $B_m(x)$, $C_m(x)$, and $D_m(x)$) to be the generating function for the number of nodes in the subtrees $\T(B;0)$ (respectively, $\T(B;a_m)$, $\T(B;b_m)$, $\T(B;c_m)$, and $\T(B;d_m)$), where its root is at level $0$. Hence, by the rules of the tree $\mathcal{T}(B)$, we have
\begin{align}
R(x)&=x+xA_0(x)+xD_1(x),\label{eqAA1}\\
A_m(x)&=x+x\sum_{j=0}^{m+1}B_j(x),\quad m\geq0,\label{eqAA2}\\
B_m(x)&=x+xA_m(x)+x\sum_{j=0}^mB_j(x),\quad m\geq0,\label{eqAA3}\\
C_m(x)&=x+xA_m(x)+x\sum_{j=1}^{m+1}D_j(x),\quad m\geq0,\label{eqAA4}\\
D_m(x)&=x+xB_m(x)+xC_m(x)+x\sum_{j=1}^mD_j(x),\quad m\geq0.\label{eqAA5}
\end{align}
\begin{figure}[t]
		{\footnotesize
			\begin{forest}
				for tree={fit=band,}
				[0[00,[001,[0011] [0012] [0013]] [002,[0022] [0023]]] [01,[010,[0101] [0102] [0103]] [011,[0110] [0112][0113]][012,[0120] [0122] [0123]]]]
		\end{forest}}
		%\caption{First three levels of $\T(\{000,021\})$}\label{figT2}
	%\end{figure}

	%\begin{figure}[htp]
		{\footnotesize
			\begin{forest}
				for tree={fit=band,}
				[0[00,[001,[0011] [001] [0]] [0,[00] [0]]] [01,[001,[0011] [001] [0]] [011,[0011] [0112][01]][01,[001] [011] [01]]]]
		\end{forest}}
		\caption{First three levels of $\T(\{000,021\})$ and $\T[\{000,021\}]$}\label{figT3}
	\end{figure}
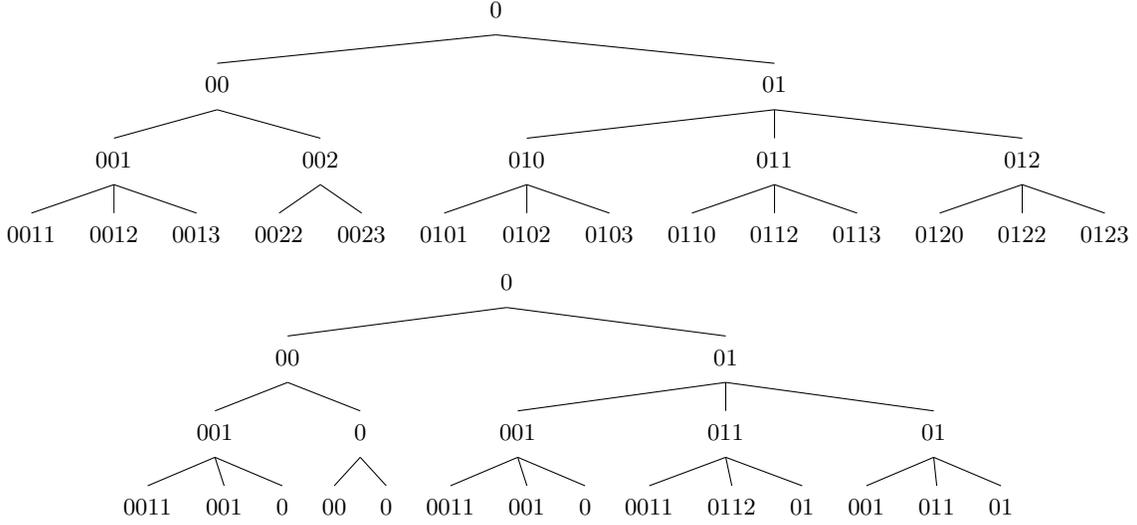
We define $A(x,u)=\sum_{m\geq0}A_m(x)u^m$, $B(x,u)=\sum_{m\geq0}B_m(x)u^m$,
$C(x,u)=\sum_{m\geq1}C_m(x)u^{m-1}$, and $D(x,u)=\sum_{m\geq1}D_m(x)u^{m-1}$.
Hence, \eqref{eqAA1}-\eqref{eqAA5} can be written as
\begin{align}
R(x)&=x+xA(x,0)+xD(x,0),\label{eqAA6}\\
A(x,u)&=\frac{x}{1-u}+\frac{x}{u}(B(x,u)-B(x,0))+\frac{x}{1-u}B(x,u),\label{eqAA7}\\
B(x,u)&=\frac{x}{1-u}+xA(x,u)+\frac{x}{1-u}B(x,u),\label{eqAA8}\\
C(x,u)&=\frac{x}{1-u}+\frac{x}{u}(A(x,u)+D(x,u)-A(x,0)-D(x,0))+\frac{x}{1-u}D(x,u),\label{eqAA9}\\
D(x,u)&=\frac{x}{1-u}+\frac{x}{u}(B(x,u)-B(x,0))+xC(x,u)+\frac{x}{1-u}D(x,u).\label{eqAA10}
\end{align}
By \eqref{eqAA7}-\eqref{eqAA8}, we have
\begin{align}\label{eqAA11}
\frac{(1-x)u-x^2-u^2}{u(1-u-x)}A(x,u)=-\frac{x^2}{u(1-x)}A(x,0)+\frac{x}{(1-u-x)(1-x)}.
\end{align}
In this case for the kernel method, if we assume that $u=x^2M(x)$, where $M(x)=\frac{1-x-\sqrt{1-2x-3x^2}}{2x^2}$ is the generating function for the Motzkin numbers, see \cite[Sequence A001006]{Slo}, then we obtain
$$A(x,0)=\frac{xM(x)}{1-x-x^2M(x)}=xM^2(x).$$
By \eqref{eqAA11}
\begin{align}\label{eqfAu}
A(x,u)&=\frac{x(ux+x^2-x)A(x,0)+u)}{(1-x)(u(1-x)-x^2-u^2)}=\frac{xM(x)(x^2M(x)-u)}{u^2+(x-1)u+x^2}.
\end{align}
Thus, by \eqref{eqAA8}
\begin{align}\label{eqfBu}
B(x,u)&=\frac{x(1+(1-u)A(x,u))}{1-u-x}=\frac{xM(x)((u+x)x^2M(x)-u+x^2)}{u^2+u(x-1)+x^2}.
\end{align}
By substituting \eqref{eqAA9} into \eqref{eqAA10} with using \eqref{eqfAu}-\eqref{eqfBu}, we obtain
\begin{align}\label{eqfDu}
&\frac{((1-x)u-x^2-u^2)^2}{u(1-u)}D(x,u)+\frac{x^2((1-x)u-x^2-u^2)}{u}D(x,0)\nonumber\\
&=\frac{xM(x)((x^2-x-1)u^2+(x^3-3x^2+1)u+x^4+x^2)}
{1-u}\nonumber\\
&+\frac{x^3M^2(x)((x-1)u^2+(x^2-x+3)u+x^3+x^2+x-2)}{1-u}.
\end{align}
By differentiating this equation respect to $u$ and taking $u=x^2M(x)$, after some simple algebraic simplifications, we get
$$D(x,0)=\frac{x((x^2+2x-2)M(x)-x+1)((1-x)M(x)-2)}{(1+x)(1-3x)}.$$
Hence, by \eqref{eqAA6}, we obtain the following result.
\begin{theorem}\label{thAA2}
The generating function $R(x)=\sum_{n\geq0}|I_n(000,021)|x^{n+1}$ is given by
\begin{align*}
&\frac{3x^3+x^2-3x+1}{2x^2\sqrt{(1+x)(1-3x)}}-\frac{(1-x)^2}{2x^2}\\
&=x+2x^2+5x^3+14x^4+39x^5+111x^6+317x^7+911x^8+2627x^9+7600x^{10}+\cdots.
\end{align*}
Moreover, by Sequence A002426 in \cite{Slo}, we get for all $n\geq1$,
$$|I_n(000,021)|=\frac{1}{2}(3a_{n-1}+a_n-3a_{n+1}+a_{n+2}),$$
where $a_n=\sum_{k=0}^n(-1)^{n-k}\binom{n}{k}\binom{2k}{k}$.
\end{theorem}

\subsection{Class $I_n(100,021)$ and $I_n(110,021)$}\label{sec100-021}
In this section, we will provide the rules for the  generating trees $\T[\{100,021\}]$ and $\T[\{110,021\}]$. The generating trees show that these two classes are equinumerous, and also lead to an exact enumerating formula.  When we apply our algorithm to the class $I_n(100,021)$, we obtain the following rules for $\T[\{100,021\}]$:
\begin{align*}
\mbox{Root: }a_0,\,\,\mbox{Rules: }&
a_m\rightsquigarrow a_{m+1}b_{m}b_{m-1}\cdots b_1,\,\,
b_m\rightsquigarrow c_mb_{m+1}b_{m}\cdots b_1,\quad m\geq1,\\
&c_m\rightsquigarrow c_{m+1}d_{m}d_{m-1}\cdots d_1 e,\,\,
d_m\rightsquigarrow d_{m+1}d_m\cdots d_1 e,\quad m\geq1,\\
&e\rightsquigarrow d_1 e,
\end{align*}
where $a_m=0^m$, $b_m=0^m1$, $c_m=0^m10$, $d_m=0^m102$ and $e=0103$.
From a very similar argument presented in the section~\ref{sec000-021}, it follows that the number of nodes at level $n$ (the root is at level $0$) in $\mathcal{T}[\{100,021\}]$ is equal to the number of inversion sequences in $I_n(100,021)$.

Next, we will apply our algorithm to the class $I_n(110,021)$, and obtain the rules for the generating tree $\mathcal{T}[\{110,021\}]$:
\begin{align*}
\mbox{Root: }a_0,\,\,\mbox{Rules: }&
a_m\rightsquigarrow a_{m+1}b_{m}b_{m-1}\cdots b_1,\,\,
b_m\rightsquigarrow c_mb_{m+1}b_{m}\cdots b_1,\quad m\geq1,\\
&c_m\rightsquigarrow c_{m+1}d_{m}d_{m-1}\cdots d_1 e,\,\,
d_m\rightsquigarrow d_{m+1}d_m\cdots d_1 e,\quad m\geq1,\\
&e\rightsquigarrow d_1 e,
\end{align*}
where $a_m=0^m$, $b_m=0^m1$, $c_m=0^m11$, $d_m=0^m112$ and $e=0113$.
The number of nodes at level $n$ in $\mathcal{T}[\{110,021\}]$ equals the number of inversion sequences in $I_n(110,021)$.
From the above generating tree rules, we have that
\begin{corollary}
For all $n\geq0$, $|I_n(100,021)|=|I_n(110,021)|$.
\end{corollary}

As in the previous subsection, after translating the rules into a system of functional equations and then solving for $\sum_{n\geq0}|I_n(100,021)|x^{n+1}$, we obtain the following result, the details are available in an earlier version of the present paper  \cite{arx1}.
\begin{theorem}\label{thCC3}
The generating function $R(x)=\sum_{n\geq0}|I_n(100,021)|x^{n+1}$ is given by
\begin{align*}
&\frac{(1-3x)^2}{2x^2\sqrt{1-4x}}-\frac{(1-3x)(1-x)}{2x^2}.
\end{align*}
Moreover, $|I_n(100,021)|=\frac{n^2+n+6}{2(n+3)(n+2)}\binom{2n+2}{n+1}$ for all $n\geq 0$.
\end{theorem}

\subsection{Class $I_n(102,021)$}
Let $B=\{102,021\}$. We will apply our algorithm to $I_n(B)$ and characterize the generating tree $\T(B)$ that leads to an exact enumerating formula for this class. We have that the generating tree $\mathcal{T}[B]$ is given by
\begin{align*}
\mbox{Root: }a_0,\,\,\mbox{Rules: }&
a_m\rightsquigarrow a_{m+1}b_{m}b_{m-1}\cdots b_1,\,\,
b_m\rightsquigarrow c_mb_{m+1}d_{m}\cdots d_2 e,\quad m\geq1,\\
&c_m\rightsquigarrow c_{m+1}c_{m+1},\,\,
d_m\rightsquigarrow fd_{m+1}d_m\cdots d_1,\\
&e\rightsquigarrow fd_2 e,\,\,f\rightsquigarrow f,
\end{align*}
where $a_m=0^m$, $b_m=0^m1$, $c_m=0^m10$, $d_m=0^m12$, $e=0013$, and $f=00130$.

By similar techniques used in the previous cases, we obtain the following result. For more details on the proof, see \cite{arx1}.

\begin{theorem}\label{thDD1}
The generating function $R(x)=\sum_{n\geq0}|I_n(102,021)|x^{n+1}$ is given by
\begin{align*}
&\frac{1-\sqrt{1-4x}}{2x(1-x)}+\frac{(2x^2-2x+1)(x^3-2x^2+3x-1)}{(1-x)^4(1-2x)}.
\end{align*}
Moreover, $|I_n(102,021)|=\sum_{k=0}^n\frac{1}{k+1}\binom{2k}{k}-1-\frac{1}{6}n^3-\frac{11}{6}n+2^n$.
\end{theorem}

\subsection{Class $I_n(100,012)$}
Let $B=\{100,012\}$. We start with the following lemmas.
\begin{lemma}\label{lemB1}
Let $m\geq1$. The generating function $B^{(1)}_m(x)$ for the number of words $\pi'$ with $n-1$ letters over alphabet $\{0,1,\ldots,m-1\}$, $n\geq1$, such that $0^mm0\pi'\in I_{n+m+1}(B)$ is given by $x(1+x)^{m-1}$.
\end{lemma}
\begin{proof}
Clearly, any inversion sequence $0^mm0\pi'\in I_{n+m+1}(B)$ can be decomposed as $0^mm0j\pi^{(j)}$ with $j=1,2,\ldots,m-1$. Note that the number of inversion sequences  $0^mm0j\pi^{(j)}\in I_{n+m+1}(B)$ equals the number of inversion sequences  $0^jj0\theta^{(j)}\in I_{n+j+1}(B)$, where $\theta^{(j)}$ is a word of length $n-2$ over alphabet $\{0,1,\ldots,j-1\}$. Hence,
$B^{(1)}_m(x)=x+x\sum_{j=1}^{m-1}B^{(1)}_j(x)$, for all $m\geq1$.
By induction on $m$, we complete the proof.
\end{proof}

\begin{lemma}\label{lemB2}
Let $m\geq1$. The generating function $B^{(2)}_m(x)$ for the number of words $\pi'$ with $n-1$ letters over alphabet $\{0,1,\ldots,m\}$, $n\geq1$, such that $0^mm0\pi'\in I_{n+m+1}(B)$ is given by $\frac{x(1+x)^{m-1}}{1-x}$.
\end{lemma}
\begin{proof}
Similar to the proof of Lemma \ref{lemB1}, we see
$B^{(2)}_m(x)=x+xB^{(2)}_m(x)+x\sum_{j=1}^{m-1}B^{(1)}_j(x)$, for all $m\geq1$. Then, by Lemma \ref{lemB1}, we complete the proof.
\end{proof}

\begin{lemma}\label{lemB3}
Let $m\geq1$. The generating function $B^{(3)}_m(x)$ for the number of words $\pi'$ with $n-1$ letters over alphabet $\{0,1,\ldots,m-1\}$, $n\geq1$, such that $0^mm\pi'\in I_{n+m}(B)$ is given by $$(m+1)x^3(1+x)^{m-2}-x(x^2-2x-1)(1+x)^{m-2}.$$
\end{lemma}
\begin{proof}
Similar to the proof of Lemma \ref{lemB1}, we see
$B^{(3)}_m(x)=x+xB^{(1)}_m(x)+x\sum_{j=1}^{m-1}B^{(3)}_j(x)$, for all $m\geq1$. Define $B^{(3)}(x,u)=\sum_{m\geq1}B^{(3)}_m(x)u^m$. By multiplying the above recurrence by $u^m$ and summing over $m\geq1$ with using Lemma \ref{lemB2}, we obtain
$$B^{(3)}(x,u)=\frac{xu(1+x-u-2ux)}{(1-u-ux)^2}.$$
Then, by finding the coefficient of $u^m$, we complete the proof.
\end{proof}

\begin{lemma}\label{lemB4}
Let $m\geq1$. The generating function $B_m(x)$ for the number of words $\pi'$ with $n-1$ letters over alphabet $\{0,1,\ldots,m\}$, $n\geq1$, such that $0^mm\pi'\in I_{n+m}(B)$ is given by $$\frac{x((m-1)x^2(1-x)+x+1)(1+x)^{m-2}}{(1-x)^2}.$$
\end{lemma}
\begin{proof}
Similar to the proof of Lemma \ref{lemB1}, we see
$B_m(x)=x+xB_m^{(2)}(x)+xB_m(x)+x\sum_{j=1}^{m-1}B^{(3)}_j(x)$, for all $m\geq1$.
By Lemmas \ref{lemB2} and \ref{lemB3}, we complete the proof.
\end{proof}

When we apply our algorithm, we obtain the following generating tree.
\begin{proposition}\label{thBB1}
Let $\mathcal{T}_m(B)$ be the generating tree for all inversion sequence $\pi=0^mm\pi'$ that avoids $\{100,012\}$.
The generating tree $\mathcal{T}[B]$ is given by
\begin{align*}
\mbox{Root: }a_1,\,\,\mbox{Rules: }a_m\rightsquigarrow a_{m+1},\mathcal{T}_1,\mathcal{T}_2,\ldots,\mathcal{T}_m,
\end{align*}
where $a_m=0^m$ with $m\geq1$.
\end{proposition}

Now, we are ready to find an explicit formula for the generating function $$R(x)=\sum_{n\geq0}|I_n(100,012)|x^{n+1}.$$

\begin{theorem}\label{thBB2}
The generating function $R(x)$ is given by
\begin{align*}
R(x)&=\frac{x(x^6-x^5-3x^4+x^3+3x^2-3x+1)}{(1-x)^3(1-x-x^2)^2}\\
&=x+2x^2+5x^3+12x^4+27x^5+56x^6+110x^7+207x^8+378x^9+675x^{10}+\cdots.
\end{align*}
Moreover, by the sequence A001629 in \cite{Slo}, for all $n\geq0$,
$$|I_n(100,012)|=\frac{(n+7)Fib_n+15Fin_{n+1}+nFib_{n+2}}{5}-1-\binom{n+2}{2},$$
where $Fib_n$ is the $n$th Fibonacci number, see sequence A000045 in \cite{Slo}.
\end{theorem}
\begin{proof}
Let $R_m(x)$ be the generating function for the number of nodes in the subtree $\mathcal{T}(B,0^m)$ of Proposition \ref{thBB1}. Hence, Proposition \ref{thBB1} and Lemma \ref{lemB4} give
$$R_m(x)=x+xR_{m+1}(x)+x\sum_{j=1}^m\frac{x((j-1)x^2(1-x)+x+1)(1+x)^{j-2}}{(1-x)^2}.$$
Define $R(x,u)=\sum_{m\geq1}R_m(x)u^{m-1}$. Then
$$R(x,u)=\frac{x}{1-u}+\frac{x}{u}(R(x,u)-R(x,0))-\frac{(ux^3-ux^2+ux+u-1)x^2}{(1-u)(ux+u-1)^2(1-x)^2}.$$
Then by applying the kernel method with taking $u=x$, we obtain
$$R(x,0)=\frac{x(x^6-x^5-3x^4+x^3+3x^2-3x+1)}{(1-x)^3(1-x-x^2)^2},$$
which completes the proof.
\end{proof}

\subsection{Class $I_n(011,201)$}
Let $B=\{011,201\}$. By applying our algorithm to $I_n(B)$, we obtain the generating tree $\T[B]$ as follows:
$$\mbox{Root: }a_1,\,\,\mbox{ Rules: }a_m\rightsquigarrow a_{m+1}a_mb_{m,2}\cdots b_{m,m},\,\, b_{m,j}\rightsquigarrow (a_{m+2-j})^2b_{m+3-j,2}\cdots b_{m,j-1}b_{m,j}\cdots b_{m,m},$$
where $a_m=0^m$ with $m\geq1$ and $b_{m,j}=0^mj$ with $2\leq j\leq m$.

We define $A_m(x)$ and $B_{m,j}(x)$ be the generating functions for the number of nodes in the subtrees $\T(B;a_m)$ and $\T(B;b_{m,j})$, respectively. Thus, the generating tree $\T[B]$, leads to
\begin{align*}
A_m(x)&=x+xA_{m+1}(x)+xA_m(x)+x(B_{m,2}(x)+\cdots+B_{m,m}(x)),\quad m\geq1,\\
B_{m,j}(x)&=x+2xA_{m+2-j}(x)+x\sum_{i=2}^{j-1}B_{m+1-j+i,i}(x)+x\sum_{i=j}^mB_{m,i}(x),\quad 2\leq j\leq m.
\end{align*}
Define $A(x,v)=\sum_{m\geq1}A_m(x)v^{m-1}$ and $B(x,v,u)=\sum_{m\geq2}\sum_{j=2}^mB_{m,j}(x)v^{m-2}u^{m-j}$. Then the above recurrence can be written as
\begin{align}
A(x,v)&=\frac{x}{1-v}+xA(x,v)+\frac{x}{v}(A(x,v)-A(x,0))+xvB(x,v,1),\label{eqT6a1}\\
B(x,v,u)&=\frac{x}{(1-v)(1-vu)}+\frac{2x}{uv(1-v)}(A(x,uv)-A(x,0))\nonumber\\
&+\frac{x}{u(1-v)}(B(x,v,u)-B(x,v,0))+\frac{x}{1-u}(B(x,v,u)-uB(x,uv,1)).\label{eqT6a2}
\end{align}
Then by taking $v=\frac{x}{1-x}$ into \eqref{eqT6a1}, we obtain
\begin{theorem}\label{thm011201}
The generating function $\sum_{n\geq0}|I_n(011,201)|x^{n+1}$ is equal to $A(x,0)$ that satisfies the following functional equation
$$A(x,0)=\frac{x}{1-2x}+\frac{x^2}{(1-x)^2}B(x,x/(1-x),1).$$
\end{theorem}
Applying this theorem, we obtain that the first $20$ terms of $A(x,0)$ as 1, 2, 5, 15, 51, 189, 746, 3091, 13311, 59146, 269701, 1256820, 5966001, 28773252, 140695923, 696332678, 3483193924, 17589239130, 89575160517, 459648885327.

\subsection{Class $I_n(120,210)$}
Let $B=\{120,210\}$. By applying our algorithm to $I_n(B)$, we obtain the the generating tree $\T[B]$ as follows:
\begin{align*}
&\mbox{Root: }a_1,\\
&\mbox{ Rules: } a_m\rightsquigarrow a_{m+1}b_{m,1}\cdots b_{m,m},\,\, b_{m,j}\rightsquigarrow b_{m+1,j}\cdots b_{m+2-j,1}b_{m+1,j}b_{m+1-j,1}\cdots b_{m+1-j,m+1-j},
\end{align*}
where $a_m=0^m$ with $m\geq1$ and $b_{m,j}=0^mj$ with $1\leq j\leq m$. It is not hard to prove this indeed the generating tree $\T[B]$ by using Lemma \ref{lem1}.

We define $A_m(x)$ and $B_{m,j}(x)$ be the generating functions for the number of nodes in the subtrees $\T(B;a_m)$ and $\T(B;b_{m,j})$, respectively. Thus, the generating tree $\T[B]$, leads to
\begin{align*}
A_m(x)&=x+xA_{m+1}(x)+x(B_{m,1}(x)+\cdots+B_{m,m}(x)),\quad m\geq1,\\
B_{m,j}(x)&=x+x\sum_{i=1}^jB_{m+1-j+i,i}(x)+xB_{m+1,j}(x)+x\sum_{i=1}^{m+1-j}B_{m+1-j,i}(x),\quad 1\leq j\leq m.
\end{align*}
We define $A(x,v)=\sum_{m\geq1}A_m(x)v^{m-1}$ and $B(x,v,u)=\sum_{m\geq1}\sum_{j=1}^mB_{m,j}(x)v^{m-1}u^{m-j}$. Then the above recurrence can be written as
\begin{align}
A(x,v)&=\frac{x}{1-v}+\frac{x}{v}(A(x,v)-A(x,0))+xB(x,v,1),\label{eqT7a1}\\
B(x,v,u)&=\frac{x}{(1-v)(1-vu)}+\frac{x}{uv(1-v)}(B(x,v,u)-B(x,v,0))\nonumber\\
&+\frac{x}{uv}(B(x,v,u)-B(x,v,0))+\frac{x}{1-v}B(x,uv,1).\label{eqT7a2}
\end{align}
Then by taking $v=x$ into \eqref{eqT7a1}, we obtain
\begin{theorem}\label{thm120210}
The generating function $\sum_{n\geq0}|I_n(120,210)|x^{n+1}$ is equal to $A(x,0)$ that satisfies the following functional equation
$$A(x,0)=\frac{x}{1-x}+xB(x,x,1).$$
\end{theorem}
Applying this theorem, we obtain that the first $20$ terms of $A(x,0)$ as 1, 2, 6, 23, 102, 499, 2625, 14601, 84847, 510614, 3161964, 20050770, 129718404, 853689031, 5701759424, 38574689104, 263936457042, 1824032887177, 12718193293888, 89386742081688.

\section{Restricted growth sequences}
In the previous sections, we showed that our algorithmic approach based on generating trees can solve many enumerative questions for inversion sequences with pattern restrictions. This approach can be modified to include enumerative results for restricted growth sequences. A sequence of positive integers $\pi=\pi_1\pi_2\cdots \pi_n$ is called
a {\em restricted growth sequence} of length $n$ if
$\pi_1 = 1$ and $\pi_{j+1} \leq 1 + \max\{\pi_1,\cdots,\pi_j\}$ for all $1\leq j <n$.
There is a bijection between these sequences and canonical set partitions.
A \textit{set partition} of a set $A$ is a collection of non-empty disjoint subsets,
called \textit{blocks}, whose union is the set $A$.
A \textit{$k$-set partition} is a set partition $\Pi$ with $k$ blocks and it is denoted by $\Pi = A_1|A_2|\cdots|A_k$.
A $k$-set partition $A_1|A_2|\cdots |A_k$ is said to be in \textit{standard form}
if the blocks $A_i$ are labeled in such a way that $\min A_1 < \min A_2<\cdots< \min A_k$.
The set partition $\Pi= A_1|A_2|\cdots|A_k$ can be represented equivalently by the \textit{canonical sequential form} (see \cite{Mb}) $\pi_1\pi_2\ldots\pi_n$, where $\pi_i\in [n]=\{1,2,\ldots,n\}$ and $i\in A_{\pi_i}$ for all $i$ .
It is easy to verify that a word $\pi\in [k]^n$ is a canonical representation of a $k$-set partition of $[n]$ in standard form
if and only if it is a restricted growth sequence, see \cite{Mb} and references therein.
Henceforth we identify set partitions with their canonical representations.
We denote by $\PP_n$ the set of all restricted growth sequences of length $n$,
and denote by $\PP_{n,k}$ the set of all restricted growth sequences of length $n$ with maximal letter $k$. Similarly to $I_n(B)$, we denote $\PP_n(B)$ be the set of all restricted growth sequences that avoid  all the patterns in $B$.

For given set of pattern $B$, we will construct a pattern-avoidance tree $\T(B)$ for the class of pattern-avoiding restricted growth sequences ${\PP}_B=\cup_{n\geq0}{\PP}_n(B)$. In the case of restricted growth sequences we define the root to be $1$ and the children of $\pi_1\pi_2\ldots\pi_{n-1}\in\PP_n(B)$ are obtained from the set $\{\pi_1\pi_2\ldots\pi_{n-1}j\mid 1\leq j\leq \max\{\pi_1\pi_2\ldots\pi_{n-1}\}+1\}$.

As in Section 2, let $\T(B;\pi)$ denote the subtree consisting of the restricted growth sequence $\pi$ as the root and its descendants in $\T(B)$. Then we define an equivalence relation on nodes $\pi,\pi'$ of $\T(B)$  whenever $\T(B;\pi)\cong\T(B;\pi')$ in the sense of plane trees.
Thus, by taking generating trees $\T(B)$ for restricted growth sequences, our algorithm as described in Section 2 can be reduced to an algorithm for finding $\T[B]$, which is the same tree $\T(B)$ where we replace each node by its equivalence class label.
For instance, if  $B=\{1212\}$, our algorithm with $D=4$ leads us to guess that the generating tree $T[\{1212\}]$ is given by
\begin{align*}
&\mbox{Root: }1,\,\mbox{Rules: }
12\cdots m\rightsquigarrow1,12,\cdots, 12\cdots(m+1).
\end{align*}
Note that by similar technique as in the proof of Lemma \ref{lem1}, we see that Lemma \ref{lem1} holds for the case of restricted growth sequences as well.

As in inversion sequences, in the next subsections, we present some applications to our algorithm for restricted growth sequences.
\subsection{Pattern $1122$}
We say that two restricted growth sequences $\pi$ and $\pi'$ are equivalent if $|{\PP}_n(\pi)|=|{\PP}_n(\pi')|$ for each $n$. In \cite{JM}, Jel\'inek and Mansour showed that there are 5 different classes of avoiding a pattern of length four, that is, $1231$, $1212$, $1122$, $1112$, and $1111$. They also suggested formulas for all the classes except the class $1122$. Our algorithm suggests the following result for the class $1122$.
\begin{lemma}\label{lem1122}
Let $a_k=12\cdots k$ and $b_{k,j}=12\cdots kj$ for $j=1,2,\ldots,k$.
Then,  the generating tree $T[\{1122\}]$ is given by
\begin{align*}
\mbox{Root}:&\,\,a_1,\,\mbox{Rules}:\, a_k\rightsquigarrow b_{k,1}b_{k,2}\cdots b_{k,k}a_{k+1},\quad b_{k,j}\rightsquigarrow b_{k,1}\cdots b_{k,j}b_{k-1,j}^{k-j}b_{k,j}.
\end{align*}
\end{lemma}
\begin{proof}
We label the root by the restricted growth function $a_1=1\in\PP_1$.
Clearly, the children of $a_k$ are $b_{k,j}$ with $j=1,2,\ldots,k$ and $a_{k+1}$.
Moreover, the children of $b_{k,j}$ are  $b_{k,j}i\sim b_{k,i}$ with $i=1,2,\ldots,j$,  $b_{k,j}i\sim b_{k-1,j}$ (by removing the letters $i$) with $i=j+1,\ldots,k$, and $b_{k,j}(k+1)\sim b_{k,j}$ (by removing the letters $k+1$). This completes the proof.
\end{proof}
We can translate the tree rules to a system of equations for the corresponding generating functions. We can obtain the first terms of the sequence $|\PP_n(\{1122\})|$ as $1$, $1$, $2$, $5$, $14$, $42$, $133$, $441$, $1523$, $5456$, $20209$, $77186$, and $303296$.

\subsection{Pattern $\{12313,12323\}$}
Based on the algorithm's output, we obtain the succession rules of the generating tree $\T[\{12313,12323\}]$.
We omit the details of the proof, since it is similar to the proof of Lemma \ref{lem1122}.
\begin{lemma}\label{lemA1}
Let $a_k=12\cdots k$.
Then,  the generating tree $\T[\{12313,12323\}]$ is given by
\begin{align*}
\mbox{Root}:&\,\,a_1,\,\mbox{Rules}:\, a_1\rightsquigarrow a_1a_2,\quad a_k\rightsquigarrow a_2^2a_3\cdots a_{k+1},\quad k\geq2.
\end{align*}
\end{lemma}
Let's define $A_k(x)$ to be the generating function for the number of nodes in the subtree $\T(\{12313,12323\};a_k)$, we have $A_1(x)=x+xA_1(x)+xA_2(x)$ and $A_k(x)=x+2xA_2(x)+xA_3(x)+\cdots+xA_{k+1}(x)$. Hence,
$$A(x,v)=\frac{x}{1-v}+\frac{x}{1-v}A_2(x)+\frac{x}{1-v}A(x,v)+\frac{x}{v}(A(x,v)-A(x,0)),$$
where $A(x,v)=\sum_{k\geq2}A_k(x)v^{k-2}$. By taking $v=\frac{1-\sqrt{1-4x}}{2}$, we obtain that
$$A_2(x)=A(x,0)=\frac{4x-1+\sqrt{1-4x}}{2(1-4x)},$$
which leads to the generating function for the number of restricted growth functions of length $n$ that avoid $\{12313,12323\}$ which is equal to $$1+A_1(x)=\frac{1}{1-x}+\frac{x}{2(1-x)}\left(\frac{1}{\sqrt{1-4x}}-1\right).$$

Moreover, we also have the following result.
\begin{lemma}\label{lemA2}
Let $a_k=12\cdots k$.
Then,  the generating tree $\T[\{12313,12323,12333\}]$ is given by
\begin{align*}
\mbox{Root}:&\,\,a_1,\,\mbox{Rules}:\, a_1\rightsquigarrow a_1a_2,\quad a_2\rightsquigarrow a_2^2a_3\quad a_k\rightsquigarrow a_2^3a_3\cdots a_{k-1}a_{k+1},\quad k\geq3.
\end{align*}
\end{lemma}
As before, Lemma \ref{lemA2} leads to the assertion that the generating function for the number of restricted growth functions of length $n$ that avoid $\{12313,12323,12333\}$ is given by
$$1+\frac{x(3-9x+\sqrt{1-2x-3x^2})}{2(2-7x)(1-x)}.$$

\subsection{Pattern $12\cdots\ell1$}
By applying our algorithm for the case $T[\{12\cdots\ell1\}]$, we see that the generating tree $T[\{12\cdots\ell1\}]$ is given by
\begin{align*}
\mbox{Root}:&\,\,1,\,\mbox{Rules}:\, a_k\rightsquigarrow (a_k)^k,a_{k+1},\mbox{ for }k=1,2,\ldots,\ell-2,\\
&\qquad\quad\qquad\, a_{\ell-1}\rightsquigarrow(a_{\ell-1})^{\ell}.
\end{align*}
where $a_k=12\cdots k$. We can easily verify this. We label the root by the restricted growth function $1\in P_1$. Clearly, the children of $a_k$ are $a_kj\sim a_k$ for $j=1,2,\ldots,k$ and $a_{k+1}$, where $k\leq \ell-2$. Thus it remains to find the children of $a_{\ell-1}$, which are  $a_{\ell-1}j\sim a_{\ell-1}$ with $j=1,2,\ldots,\ell-1$ and $a_{\ell-1}\ell\sim a_{\ell-1}$ (by removing the letter $1$ because it avoids $12\cdots\ell1$). This completes the proof.

To find a formula for the generating function $\sum_{n\geq1}|\PP_n(\{12\cdots\ell1\})|x^n$, we define $A_m(x)$ to be the generating function for the number of nodes in the subtrees $\T(\{12\cdots\ell\};a_m)$. Hence, by the rules of $\T[\{12\cdots\ell1\}]$, we have that $A_k(x)=x+kxA_k(x)+xA_{k+1}(x)$ with $k=1,2,\ldots,\ell-2$, and $A_{\ell-1}(x)=x+\ell xA_{\ell-1}(x)$. By induction on $k$, we have
$$A_k(x)=\frac{x^{\ell-k}(1-(\ell-1)x)}{\prod_{j=k}^\ell(1-jx)}
+\sum_{i=1}^{\ell-1-k}\frac{x^i}{\prod_{j=k}^{k-1+i}(1-jx)},$$
which implies the following result.
\begin{theorem}
The generating function $\sum_{n\geq1}|\PP_n(\{12\cdots\ell1\})|x^n$ is given by
$$\frac{x^{\ell-1}(1-(\ell-1)x)}{\prod_{j=1}^\ell(1-jx)}
+\sum_{i=1}^{\ell-2}\frac{x^i}{\prod_{j=1}^i(1-jx)}.$$
\end{theorem}

\noindent{\bf Acknowledgement}
We thank one anonymous referee whose suggestions improved the presentation of the paper.
%---------------------------------------------------------------

\end{document}